\documentclass[11pt,a4paper]{article}
\usepackage[utf8]{inputenc}
\usepackage[T1]{fontenc}
\usepackage[english]{babel}
\usepackage{lmodern,microtype}
\usepackage{amsmath,amssymb,amsthm,mathtools} 
\usepackage{cite} 
\usepackage{eucal,enumitem}
\usepackage{comment}

\usepackage[nottoc]{tocbibind} 
\usepackage{tikz,mdframed} 
\usepackage[colorlinks,
  linkcolor=red!60!black,%
  citecolor=green!40!black,%
  urlcolor=black,%
]{hyperref} 

\usepackage[left=3.0cm,right=3.0cm,top=1.5cm,bottom=1.6cm]{geometry}
\usepackage{changepage} 
\usepackage{marginnote} 
\usepackage{zref-savepos,zref-abspage} 
\usepackage{expl3} 

\usepackage[notref,notcite,final%
]{showkeys} 

\usepackage[mathlines]{lineno}
\usepackage{etoolbox} 

\newcommand*\linenomathpatch[1]{%
   \expandafter\pretocmd\csname #1\endcsname {\linenomath}{}{}%
   \expandafter\pretocmd\csname #1*\endcsname{\linenomath}{}{}%
   \expandafter\apptocmd\csname end#1\endcsname {\endlinenomath}{}{}%
   \expandafter\apptocmd\csname end#1*\endcsname{\endlinenomath}{}{}%
 }
\newcommand*\linenomathpatchAMS[1]{%
    \expandafter\pretocmd\csname #1\endcsname {\linenomathAMS}{}{}%
    \expandafter\pretocmd\csname #1*\endcsname{\linenomathAMS}{}{}%
    \expandafter\apptocmd\csname end#1\endcsname {\endlinenomath}{}{}%
    \expandafter\apptocmd\csname end#1*\endcsname{\endlinenomath}{}{}%
}

\expandafter\ifx\linenomath\linenomathWithnumbers
\let\linenomathAMS\linenomathWithnumbers
\patchcmd\linenomathAMS{\advance\postdisplaypenalty\linenopenalty}{}{}{}
\else
\let\linenomathAMS\linenomathNonumbers
\fi

\linenomathpatchAMS{gather}
\linenomathpatchAMS{multline}
\linenomathpatchAMS{align}
\linenomathpatchAMS{alignat}
\linenomathpatchAMS{flalign}
\linenomathpatch{equation}

\newcommand{\Aut}[0]{\text{Aut}}

\colorlet{tn/color/theorem}{green!60!black}
\colorlet{tn/color/claim}{gray!60!white}
\colorlet{tn/color/definition}{black}
\colorlet{tn/color/question}{orange!60!white}
\colorlet{tn/color/conjecture}{orange!60!white}
\colorlet{tn/color/defi}{red!50!black}

\theoremstyle{definition}

\newtheoremstyle{slthm}
{0pt}
{0pt}
{\normalfont}
{}
{\bfseries\footnotesize}
{.\!}
{.6em}
{\thmname{#1}\thmnumber{ #2}\thmnote{\,\,--\ignorespaces#3}}
\theoremstyle{slthm}

\newlength{\mythmbar}
\newlength{\myinsep}
\newlength{\myleftinsep}
\newlength{\mylen}
\newlength{\thmskip}
\mdfsetup{skipabove=\thmskip,skipbelow=\thmskip,usetwoside=false,nobreak}

\newcommand{\tndefmdstyle}[2]{%
  \mdfdefinestyle{tn/#1}{%
    linewidth=\mythmbar,%
    linecolor=tn/color/#2,%
    bottomline=false,%
    topline=false,%
    innerleftmargin=\myleftinsep,%
    leftmargin=-\myinsep,%
    innerrightmargin=\myinsep,%
    innertopmargin=1pt,%
    innerbottommargin=0pt,%
    rightmargin=\myinsep,%
    skipabove=\bigskipamount,%
    skipbelow=3pt,%
    userdefinedwidth=\mylen}}

\tndefmdstyle{base}{theorem}
\tndefmdstyle{claim}{claim}
\tndefmdstyle{definition}{definition}
\tndefmdstyle{question}{question}

\newmdtheoremenv[style=tn/base]{theorem}{Theorem}

\newcommand{\tnmdenv}[3]{%
  \newmdtheoremenv[style=#1]{#2}[theorem]{#3}}

\tnmdenv{tn/base}{lemma}{Lemma}
\tnmdenv{tn/base}{proposition}{Proposition}
\tnmdenv{tn/claim}{claim}{Claim}
\tnmdenv{tn/claim}{example}{Example}
\tnmdenv{tn/base}{corollary}{Corollary}
\tnmdenv{tn/base}{fact}{Fact}
\tnmdenv{tn/question}{question}{Question}
\tnmdenv{tn/question}{problem}{Problem}
\tnmdenv{tn/question}{conjecture}{Conjecture}
\tnmdenv{tn/definition}{definition}{Definition}
\tnmdenv{tn/definition}{remark}{Remark}
\tnmdenv{tn/definition}{hole}{Hole}

\newtheoremstyle{tnUnnumb}
{0pt}
{0pt}
{\normalfont}
{}
{\bfseries\footnotesize}
{.}
{.5em}
{\thmnote{#3}}

\theoremstyle{tnUnnumb}
\newmdtheoremenv[style=tn/base]{unnumtheorem}{Theorem}
\newmdtheoremenv[style=tn/base]{unnumconjecture}{Conjecture}

\newcommand{\newproofenv}[3]{%
  \newenvironment{#1}[1][#2]{%
    \renewcommand{\qedsymbol}{\hfill{#3}}\begin{proof}[#2]%
  }{%
    \end{proof}\renewcommand{\qedsymbol}{\oldqed}%
  }%
}

\newproofenv{claimproof}{Proof}{$\diamond$}
\newproofenv{proofsketch}{Proof sketch}{$\bigtriangleup$}

\newcommand{\defistyle}[1]{\textcolor{definition}{\textsl{#1}}}
\newcommand{\defi}[1]{%
  \sidepar{{\tiny#1}}%
  {\defistyle{#1}}}
\newcommand{\sgn}[0]{\text{sgn}}

\ExplSyntaxOn\makeatletter
\int_new:N \g_tassio_note_int
\seq_new:N\g_tassio_note_seq
\renewcommand{\defi}{\@dblarg\tn@defi} 
\renewcommand{\defistyle}[1]{%
  \textcolor{tn/color/defi}{\textsl{#1}}%
}
\def\tn@defi[#1]#2{%
  \int_gincr:N \g_tassio_note_int
  \bool_if:nTF
  {
   \int_compare_p:n
    {
     \zposy{tassionotepos\int_eval:n{\g_tassio_note_int}}
     =
     \zposy{tassionotepos\int_eval:n{\g_tassio_note_int +1}}
    }
   &&
   \int_compare_p:n
   {
    \zref@extractdefault { tassionotepage\int_eval:n{\g_tassio_note_int } }{abspage}{-1}
    =
    \zref@extractdefault { tassionotepage\int_eval:n{\g_tassio_note_int +1} }{abspage}{-1}
   }
  }
   {
    \seq_gput_right:Nn \g_tassio_note_seq {#1}
   }
   {
    \seq_gput_right:Nn \g_tassio_note_seq {#1}
    \marginnote{\fontsize{7pt}{5pt}\selectfont%
      \baselineskip=.6\baselineskip 
      \lineskip=-1.2pt 
      \seq_use:Nn \g_tassio_note_seq {,\ }}
    \seq_gclear:N \g_tassio_note_seq
   }
  \zref@label {tassionotepage\int_use:N\g_tassio_note_int}\zsaveposy {tassionotepos\int_use:N\g_tassio_note_int}
  \defistyle{#2}}

\ExplSyntaxOff

\def\tn@defi[#1]#2{{\defistyle{#2}}}

\makeatother

\makeatletter
\renewcommand{\@biblabel}[1]{\textcolor{gray}{[\,}#1\textcolor{gray}{\,]}}
  \renewcommand\@openbib@code{
    \setlength\labelwidth{2cm}%
    \setlength{\itemindent}{0cm}%
    \setlength{\leftmargin}{0cm}%
    \setlength\labelsep{.8em}%
    }
\makeatother

\newcommand{\out}{\mathrm{out}}

\newcommand{\UG}[0]{\textrm{UG}}

\definecolor{mygreen}{rgb}{0.01, 0.75, 0.24}

\let\emptyset\varnothing

\def\emb{\mathop{\text{\rm emb}}\nolimits}

\def\({\left(}
\def\){\right)}
\let\:=\colon

\let\geq\geqslant
\let\leq\leqslant

\newcommand{\Aplain}[1]{\ensuremath{(A_{#1})}}
\newcommand{\Bplain}[1]{\ensuremath{(B_{#1})}}
\newcommand{\Cplain}[1]{\ensuremath{(C_{#1})}}
\newcommand{\A}[1]{\hyperref[config_A]{\Aplain{#1}}}
\newcommand{\B}[1]{\hyperref[config_B]{\Bplain{#1}}}
\newcommand{\C}[1]{\hyperref[config_B]{\Cplain{#1}}}

\begin{document}
\begin{center} 
  {\Large\noindent\textbf{On converse invariant trees of diameter four}}

  \medskip
  Fernando Afonso\footnote{Instituto de Matemática e Estatística, Universidade de São Paulo, Brazil \texttt{fernando.s.afonso@usp.br}}
  \qquad
  Lucas Colucci\footnote{Instituto de Matemática e Estatística, Universidade de São Paulo, Brazil \texttt{tnaia@crm.cat},
research supported by FAPESB (EDITAL FAPESB Nº 012/2022 - UNIVERSAL -
NºAPP0044/2023). FAPESB is the Bahia Research Foundation.}
  \qquad
  T\'assio Naia\footnote{Centre de Recerca Matem\`atica, Edifici C, Campus Bellaterra, 08193, Spain, \texttt{tnaia@crm.cat},
research supported by the European Union's Horizon Europe Marie Sk\l{}odowska-Curie grant PARTIORI -- project number 101207083 \raisebox{-.2ex}{\includegraphics[height=2.2ex]{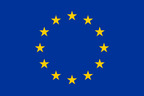}}.}
\end{center} 

\begin{center} 
  
 \begin{abstract}
Let $D$ be an oriented graph, and let $f_T(D)$ denote the number of
copies of $D$ in a tournament $T$. We say that $D$ is
\emph{converse invariant} if
$f_T(D)=f_T(\overline D)$
for every tournament $T$, where $\overline D$ is obtained from $D$
by reversing all arcs. Ai, Gutin, Lei, Yeo, and Zhou introduced a
digraph polynomial for studying this property and conjectured that an
orientation of a tree of maximum degree at least $3$ is converse
invariant if and only if it is self-converse or can be obtained
recursively by bridge-mirroring from an orientation of a path.

We disprove this conjecture. More precisely, we characterize
converse-invariant orientations of trees of diameter four and exhibit
non-self-converse examples that do not arise from the recursive
bridge-mirroring construction. To prove the classification, we introduce a multilinear
polynomial $P_D$ encoding the difference
$f_T(D)-f_T(\overline D)$ over all tournaments $T$, and we give a
coefficient formula for $P_D$ as a signed sum over copies of
subgraphs of the underlying graph of $D$. This polynomial method yields parity
obstructions, gives new proofs that oriented paths and cycles are
converse invariant, and provides the main tool for the diameter-four
classification.
\end{abstract}

\end{center} 

\section{Introduction}
Let $D$ be an oriented graph, i.e., a simple digraph without loops and having no opposite pairs of directed edges, and $T$ be a tournament. We denote by $f_T(D)$ the number of copies of $D$ in $T$, i.e., the number of subdigraphs of $T$ isomorphic to $D$, and by $\overline{D}$ the \emph{converse oriented graph} of $D$, i.e., the oriented graph obtained from $D$ by changing the direction of every arc of $D$. In this paper, we are interested in studying oriented graphs $D$ that satisfy $f_T(D)=f_{\overline{T}}(D)$, or, equivalently, $f_T(D)=f_T({\overline{D}})$, for every tournament $T$. We say that the oriented graph is \emph{converse invariant} if this condition holds.

An oriented graph $D$ is called \emph{self-converse} if $D$ and $\overline{D}$ are isomorphic oriented graphs. Clearly, every self-converse oriented graph is converse invariant. In 2023, El Sahili and Hanna \cite{el2023number} asked whether it is possible to characterize the converse invariant digraphs, and found the first family of converse invariant digraphs that are not necessarily self-converse:

\begin{theorem}[\ (El Sahili, Hanna) \cite{el2023number}]\label{thm:elsahili}
    Every oriented path and cycle is converse invariant.
\end{theorem}

More recently, in 2025, Ai, Gutin, Lei, Yeo, and Zhou \cite{ai2025number} classified converse invariant oriented trees of diameter at most 3:

\begin{theorem}[\ (Ai, Gutin, Lei, Yeo, Zhou) \cite{ai2025number}]\label{thm:ai}
    Let $D$ be an orientation of a tree with diameter at most three that is not a path. Then $D$ is converse invariant if and only if it is self-converse, or $D$ or its converse is isomorphic to the digraph shown in Figure \ref{fig:bmex}.
\end{theorem}

\begin{figure}[h]
    \centering
    \includegraphics[width=0.4\linewidth]{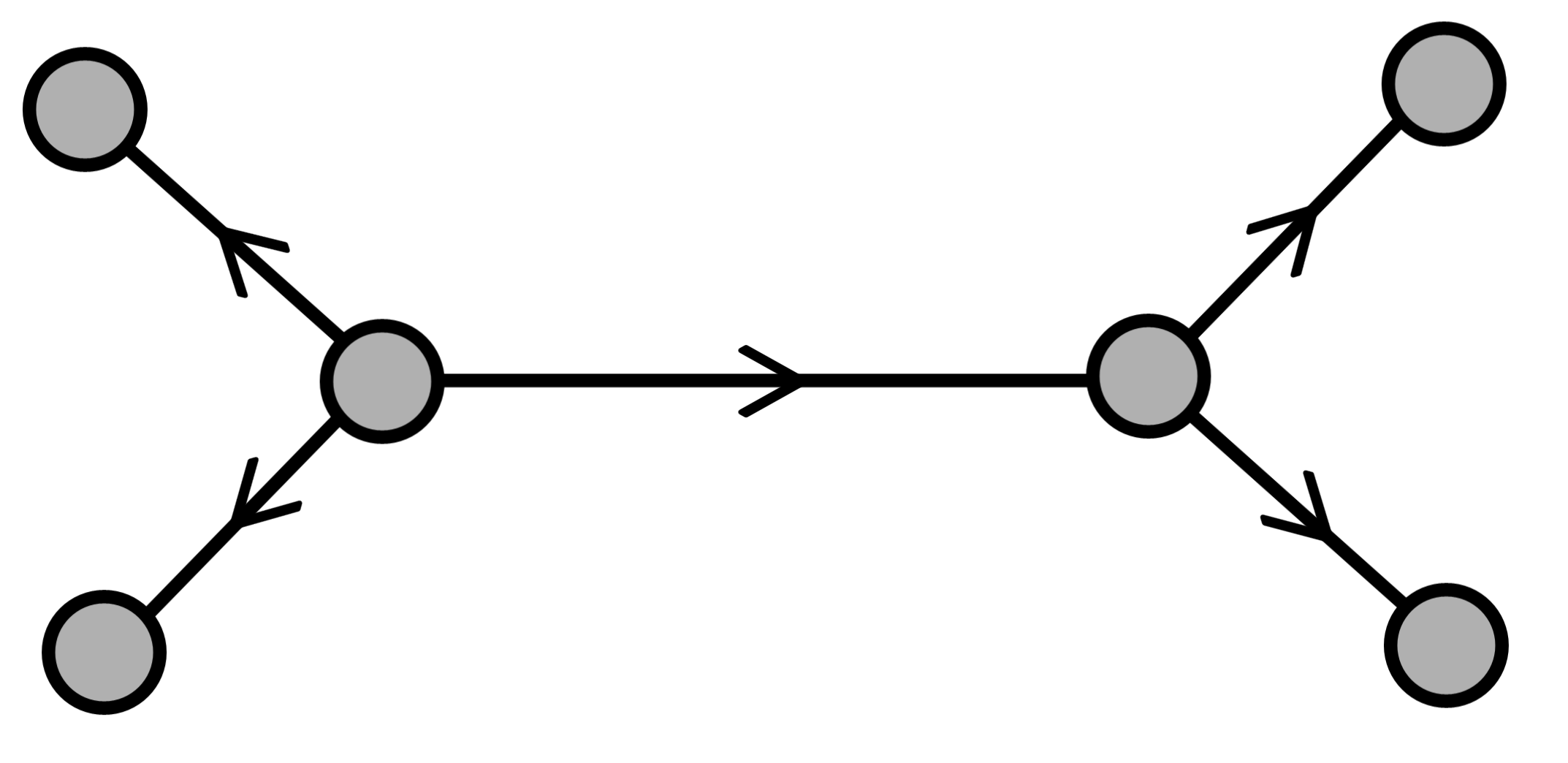}
    \caption{A converse invariant tree of diameter $3$}
    \label{fig:bmex}
\end{figure}

For a digraph $D$ and $u \in V(D)$, the \textit{bridge-mirroring} of $D$ at the vertex $u$ is the digraph obtained by taking two copies of $D$ and adding an arc between the two vertices corresponding to $u$. This operation was introduced by Zhao and Zhou \cite{zhao2020impartial}.

In \cite{ai2025number}, it was proved that the bridge-mirroring preserves the converse invariance property. Furthermore, the authors conjecture that this is essentially the only way to produce converse invariant trees:

\begin{conjecture}[\ (Ai, Gutin, Lei, Yeo, Zhou) \cite{ai2025number}]\label{conj:ai}
    Let $T$ be an orientation of a tree with maximum degree at least 3. Then $T$ is converse invariant if and only if it is self-converse or $T$ can be obtained by the bridge-mirroring operation recursively from an orientation of a path.
\end{conjecture}

In this paper, we disprove Conjecture \ref{conj:ai}, and introduce a new polynomial that encodes the orientations of a tournament that allow us to find new obstructions for a digraph to be converse invariant (Theorem \ref{thm:oddoddodd}), give new proofs of Theorems \ref{thm:elsahili} and \ref{thm:ai} and characterize the converse invariant diameter-four trees as follows:

\begin{definition}\label{def:exceptionaldiam4}
Let $h\geq 1$.  An oriented tree $T$ belongs to the family
$\mathcal E_4$ if it is constructed as follows.  Start with a path
$$
        x-u-y
$$
and attach $2h$ leaves to each of the three vertices $x,u,y$.  Orient
the $2h$ pendant edges incident with each of $x,u,y$ so that exactly
$h$ of them point away from the center and exactly $h$ point towards
the center.  Finally, orient the two edges $xu$ and $uy$ either both
towards $u$, or both away from $u$.

Equivalently, $T$ is obtained from three copies of the same balanced
oriented star by identifying their centers with the vertices of a path
of length two, and orienting the two edges of this path so that the
middle vertex is either a source or a sink.
\end{definition}

Our main result, proved in Section~\ref{subsec:trees4}, is the following classification of diameter-four converse invariant trees:

\begin{theorem}\label{thm:diam4classification}
Let $T$ be an orientation of a tree of diameter four which is not a path.
Then $T$ is converse invariant if and only if either $T$ is
self-converse or $T\in\mathcal E_4$.
\end{theorem}

\section{Preliminaries}\label{sec:prelim}

It is enough to test converse invariance on tournaments with exactly $|V(D)|$ vertices. Indeed, if the equality holds for every tournament on $|V(D)|$ vertices and $T$ is a larger tournament, then summing the equality over all $|V(D)|$-vertex subtournaments of $T$ gives $f_T(D)=f_T(\overline D)$, since each copy of $D$ uses exactly $|V(D)|$ vertices. Throughout the paper, $\UG(D)$ denotes the underlying graph of $D$, obtained from $D$ by ignoring the orientations of the edges.

Let $D$ be an oriented graph on $n$ vertices and put $e(D)=|E(D)|$. We shall associate to $D$ a multilinear polynomial $P_D$ whose non-vanishing on a point of $\{\pm1\}^{\binom n2}$ is equivalent to the existence of a tournament $T$ with $f_T(D)\neq f_T(\overline D)$.

Index the vertices of both $D$ and a tournament $T$ by $[n]=\{1,\dots,n\}$. For each unordered pair $ij$, let $x_{ij}=x_{ji}$ be $1$ if the edge of $T$ is directed from $\min(i,j)$ to $\max(i,j)$, and $-1$ otherwise. Similarly, for each edge $ij\in E(D)$, let $y_{ij}=y_{ji}$ be $1$ if the arc of $D$ is directed from $\min(i,j)$ to $\max(i,j)$, and $-1$ otherwise. When $D$ is acyclic, we shall always choose a topological labeling of $D$, so that $y_{ij}=1$ for every edge $ij\in E(D)$.

For a permutation $\pi\in S_n$, we embed vertex $i$ of $D$ into vertex $\pi(i)$ of $T$. This embedding is a copy of $D$ if and only if
\[
        x_{\pi(i)\pi(j)}y_{ij}(-1)^{(ij)_\pi}=1
\]
for every $ij\in E(D)$, where $(ij)_\pi=0$ if $\pi(i)<\pi(j)$ and $(ij)_\pi=1$ otherwise. Therefore
\begin{equation*}\label{eq:copy-count}
    f_T(D)=\frac{1}{|\Aut(D)|}\sum_{\pi\in S_n}\prod_{ij\in E(D)}
    \frac{1+x_{\pi(i)\pi(j)}y_{ij}(-1)^{(ij)_\pi}}{2}.
\end{equation*}

We write $\emb_T(D)$ for the number of labeled embeddings of $D$ into $T$.
Thus
$$
        f_T(D)=\frac{\emb_T(D)}{|\Aut(D)|}.
$$

We define
\begin{align*}\label{eq:PD-definition}
    P_D(\underline{x})
    &\coloneqq 2^{e(D)}|\Aut(D)|\bigl(f_T(D)-f_T(\overline D)\bigr) \\
    &=\sum_{\pi\in S_n}\left(
       \prod_{ij\in E(D)}\bigl(1+x_{\pi(i)\pi(j)}y_{ij}(-1)^{(ij)_\pi}\bigr)
       -\prod_{ij\in E(D)}\bigl(1-x_{\pi(i)\pi(j)}y_{ij}(-1)^{(ij)_\pi}\bigr)
       \right).
\end{align*}
Thus $D$ is not converse invariant if and only if $P_D$ is nonzero at some point of $\{\pm1\}^{\binom n2}$.

The polynomial $P_D$ is multilinear and all its monomials have odd degree. If $H$ is a graph whose edges are pairs from $[n]$, oriented from the smaller label to the larger label, write
\[
        x^H=\prod_{ij\in E(H)}x_{ij}.
\]
For $|E(H)|$ odd, the coefficient of $x^H$ in $P_D$ is
\begin{equation}\label{eq:coef-general}
    c_H=2(n-|V(H)|)!\sum_S\sum_\alpha
       \prod_{ij\in E(S)}y_{ij}(-1)^{(ij)_\alpha},
\end{equation}
where $S$ ranges over the copies of $\UG(H)$ in $\UG(D)$, and $\alpha:V(S)\to V(H)$ ranges over the graph isomorphisms from $S$ to $\UG(H)$. For even $|E(H)|$, the coefficient is zero. In particular, if all coefficients $c_H$ vanish, then $P_D$ is the zero polynomial; conversely, every nonzero coefficient gives a tournament witnessing that $D$ is not converse invariant.
\medskip

\begin{figure}[h!]
    \centering
    \includegraphics[width=0.6\linewidth]{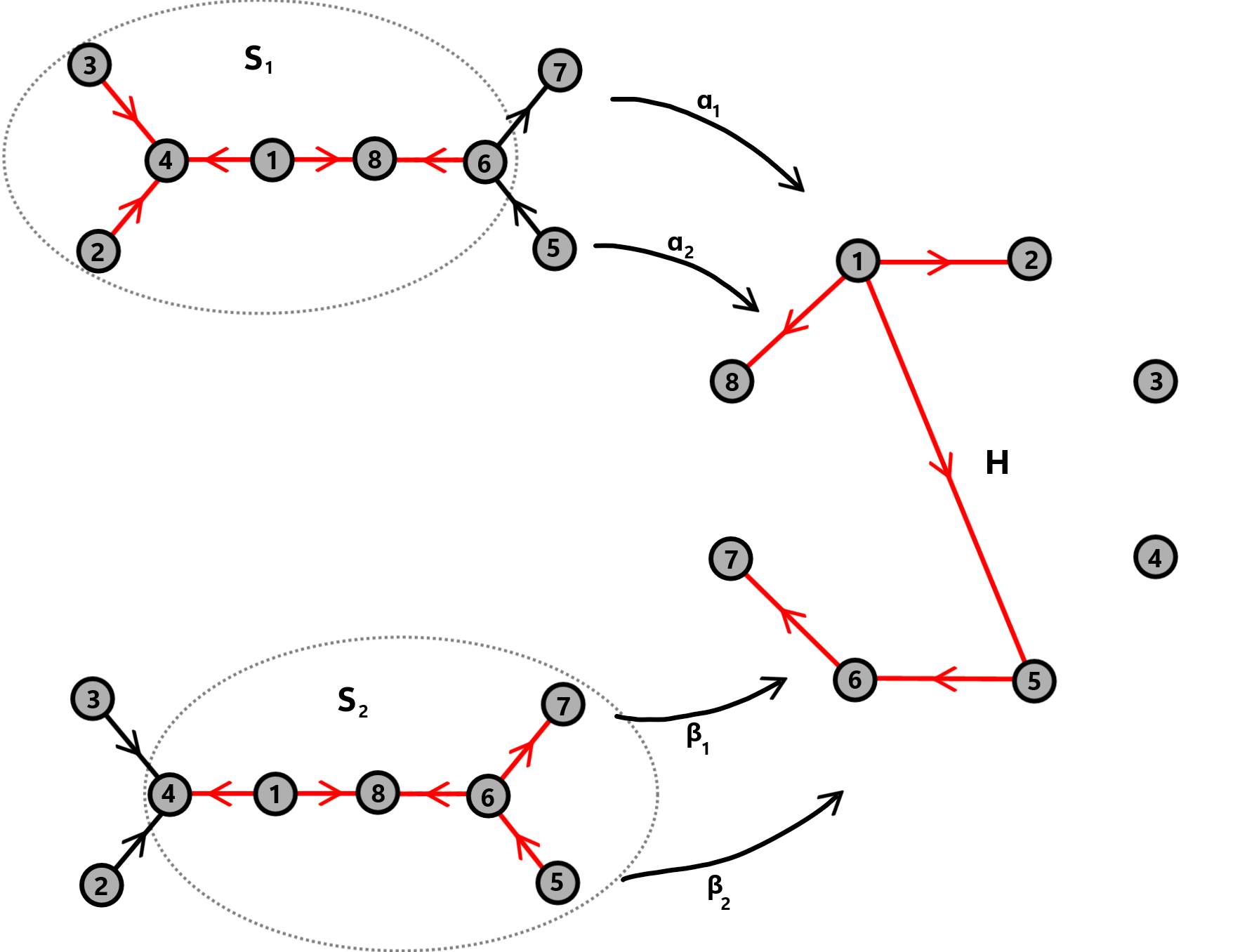}
    \caption{An example of computation of $c_H$}
    \label{fig:examplecoefficient}
\end{figure}

Figure \ref{fig:examplecoefficient} illustrates the computation of the coefficient of $x^H=x_{12}x_{18}x_{15}x_{56}x_{67}$ in the polynomial $P_D$, where $D$ is the oriented tree depicted on the left. Note that $y_{ij}=1$ for every $ij \in E(D)$. There are two copies of $\UG(H)$ in $\UG(D)$, denoted by $S_1$ and $S_2$ and highlighted in red. For each copy, there are two isomorphisms to $\UG(H)$: $\alpha_1:V(S_1) \rightarrow V(H)$, $\alpha_2:V(S_1) \rightarrow V(H)$, where 

\medskip

$\alpha_1=\binom{\ 1 \ 2 \ 3 \ 4 \ 6 \ 8 \ }{\ 5 \ 2 \ 8 \ 1 \ 7 \ 6 \ }$, $\alpha_2=\binom{\ 1 \ 2 \ 3 \ 4 \ 6 \ 8 \ }{\ 5 \ 8 \ 2 \ 1 \ 7 \ 6 \ }$, $\beta_1=\binom{\ 1 \ 4 \ 5 \ 6 \ 7 \ 8 \ }{\ 6 \ 7 \ 8 \ 1 \ 2 \ 8 \ }$ and $\beta_2=\binom{\ 1 \ 4 \ 5 \ 6 \ 7 \ 8 \ }{\ 6 \ 7 \ 8 \ 1 \ 8 \ 2 \ }$.

\medskip

We have $\prod_{ij \in E(S_1)}(-1)^{(ij)_{\alpha_{1}}} = 1$, $\prod_{ij \in E(S_1)}(-1)^{(ij)_{\alpha_{2}}} = 1 $, $\prod_{ij \in E(S_2)}(-1)^{(ij)_{\beta_{1}}} = 1 $, $\prod_{ij \in E(S_2)}(-1)^{(ij)_{\beta_{2}}} = 1 $. Hence, 

$$c_H = 2(n-|V(H)|)!\sum_{S}\sum_{\alpha}\prod_{ij \in E(S)}y_{ij}(-1)^{(ij)_\alpha} = 4(1+1+1+1) = 16.$$

\medskip

Finally, we will need the following simple fact, that one can prove by induction: a multilinear real polynomial $f$ on $m$ variables is identically zero if, and only if, $f(\underline{x}) = 0$ for every $\underline{x} \in \{\pm 1\}^m$. Hence, an oriented graph $D$ is converse invariant if, and only if, every coefficient of $P_D$ is zero.

\section{First results}

In this section we prove two consequences of the coefficient formula. The first is a parity obstruction to converse invariance. The second is a polynomial proof of Theorem \ref{thm:elsahili}.

\begin{theorem}\label{thm:oddoddodd}
    Let $D$ be an oriented graph with underlying graph $G$. Suppose that $G$ has a subgraph $H$ such that
    \begin{enumerate}
        \item $H$ has an odd number of edges;
        \item $|\Aut(H)|$ is odd;
        \item the number of copies of $H$ in $G$ is odd.
    \end{enumerate}
    Then $D$ is not converse invariant. In particular, if $G$ has an odd number of edges and $|\Aut(G)|$ is odd, then $D$ is not converse invariant.
\end{theorem}

\begin{proof}
    Let $H'\subseteq K_n$ be any labeled copy of $H$, with each edge oriented from the smaller label to the larger label. By \eqref{eq:coef-general}, the coefficient of $x^{H'}$ in $P_D$ is
    \[
        2(n-|V(H')|)!\sum_S\sum_\alpha
        \prod_{ij\in E(S)}y_{ij}(-1)^{(ij)_\alpha},
    \]
    where $S$ ranges over the copies of $H$ in $G$ and $\alpha:S\to H'$ ranges over graph isomorphisms. Each summand is $+1$ or $-1$. The number of summands is
    \[
        \#\{\text{copies of }H\text{ in }G\}\cdot |\Aut(H)|,
    \]
    which is odd by assumption. A sum of an odd number of signs cannot be zero, so this coefficient is nonzero. Hence $P_D$ is not the zero polynomial, and $D$ is not converse invariant.
\end{proof}

\begin{theorem}\label{thm:paths-cycles}
    Oriented paths and oriented cycles are converse invariant.
\end{theorem}

\begin{proof}
    By the coefficient criterion, it is enough to prove that every coefficient \eqref{eq:coef-general} is zero. We first consider the case where $D$ is an oriented path. Since $D$ is acyclic, we take a topological labeling and hence $y_{ij}=1$ for every $ij\in E(D)$.

    Let $H$ be a subgraph of $\UG(D)$ with an odd number of edges. Then $H$ is a disjoint union of paths, and at least one component, say $Q$, has an odd number of edges. Let $\rho$ be the automorphism of $H$ that reflects $Q$ and fixes every vertex outside $Q$. For every copy $S$ of $H$ in $\UG(D)$ and every isomorphism $\alpha:S\to H$, pair $\alpha$ with $\rho\circ\alpha$. This pairing has no fixed point.

    We claim that the two paired isomorphisms contribute opposite signs to \eqref{eq:coef-general}. Outside the component $\alpha^{-1}(Q)$, the two products are identical. On $Q$, write the vertices in path order as $v_1,\dots,v_{2r}$, so that $Q$ has $2r-1$ edges. The reflection pairs all edges except the middle edge $v_rv_{r+1}$; the paired edges contribute the same product before and after applying the reflection, while the middle edge is reversed. Thus
    \[
        \prod_{ij\in E(S)}(-1)^{(ij)_\alpha}
        =-
        \prod_{ij\in E(S)}(-1)^{(ij)_{\rho\circ\alpha}}.
    \]
    Hence the contributions cancel in pairs, and every coefficient is zero.

    The proof for cycles is the same, except for one case. If $H$ is a proper subgraph of a cycle, then $H$ is a disjoint union of paths, so the preceding argument applies to a path component with an odd number of edges. If $H$ is the whole cycle, then $|E(H)|$ must be odd. Pair each isomorphism with a reflection of the odd cycle whose axis passes through one vertex and the midpoint of the opposite edge. This reflection fixes exactly one edge setwise and reverses its endpoints, while all other edges are paired. Therefore the sign changes by a factor of $-1$, and the contributions again cancel. The factor $\prod_{ij\in E(S)}y_{ij}$ is independent of $\alpha$, so the argument applies to every orientation of the cycle.
\end{proof}

\section{Trees}\label{sec:trees}

In this section, we will deal with theorems concerning trees. As they do not contain cycles, every orientation of a tree corresponds to a labeling and vice versa in the natural way, i.e., the edges point towards the larger label.

\subsection{Lemmas}\label{subsec:lemmas}

In this subsection, we will prove some relevant lemmas that will be used in subsections \ref{subsec:trees3} and \ref{subsec:trees4}. We start with a definition:

\begin{definition}
    Let $S$ and $T$ be two labeled isomorphic graphs, and $\alpha:V(S) \rightarrow V(T)$ be an isomorphism. We define the \emph{sign of $\alpha$ along the edges of $S$} as

    $$
    \sgn_S(\alpha) = \prod_{ij \in E(S)}(-1)^{(ij)_\alpha}.
    $$
\end{definition}

\begin{lemma}\label{lem:constantsignstar}
    Let $S$ and $T$ be labeled stars with $|V(S)|=|V(T)|\geq 3$. Then $\sgn_S(\alpha)$ is the same for every isomorphism $\alpha:S\to T$.
\end{lemma}

\begin{proof}
    Any two isomorphisms $S\to T$ differ by a permutation of the leaves of $S$. Since such permutations are generated by transpositions, it is enough to compare $\alpha$ with the map obtained from $\alpha$ by swapping the images of two leaves.

    Let $c$ be the label of the center of $T$, and let $p<q$ be the two leaf-images being swapped. All factors in $\sgn_S(\alpha)$ are unchanged except possibly those corresponding to the two leaf-edges whose images are incident with $p$ and $q$. If $p$ and $q$ lie on the same side of $c$, neither of these two factors changes. If $p<c<q$, then both factors change sign. In both cases their product is unchanged. Hence $\sgn_S(\alpha)$ is independent of $\alpha$.
\end{proof}

\begin{definition}
    A \emph{rooted} tree $T$ is a tree with a distinguished vertex $r\in V(T)$, called the \emph{root}. A \emph{rooted isomorphism} between rooted trees is an isomorphism that maps the root of the first tree to the root of the second. For a vertex $x$ in a rooted tree, $R(x)$ denotes the subtree induced by $x$ and its descendants. Finally, $T_\out$ denotes the orientation of the rooted tree $T$ in which every edge points away from the root.
\end{definition}

\begin{lemma}\label{lem:constantsigneventree}
    Let $S$ and $T$ be isomorphic rooted labeled trees. Then $\sgn_S(\alpha)$ is the same for every rooted isomorphism $\alpha:S\to T$. In particular, if $S$ and $T$ are isomorphic trees of even diameter, then $\sgn_S(\alpha)$ is the same for every isomorphism $\alpha:S\to T$.
\end{lemma}

\begin{proof}
    We prove the first statement by induction on $|V(S)|$. The cases of one edge and of rooted stars follow from the definition and Lemma \ref{lem:constantsignstar}.

    Let $r_S$ and $r_T$ be the roots. The children of $r_S$ split into classes according to the rooted isomorphism type of the subtrees below them; the same is true for $r_T$. A rooted isomorphism is obtained by, independently for each class, matching the corresponding child-subtrees of $r_S$ with those of $r_T$, and then choosing rooted isomorphisms inside the matched subtrees.

    By the induction hypothesis, the contribution from the edges inside each matched child-subtree is independent of the chosen rooted isomorphism inside that subtree. It remains to check that permuting isomorphic child-subtrees within one class does not change the contribution of the edges incident with the parent. But these parent--child edges form a star, and Lemma \ref{lem:constantsignstar} shows that their contribution is invariant under such permutations. Therefore the full product defining $\sgn_S(\alpha)$ is independent of $\alpha$.

    If a tree has even diameter, then it has a unique central vertex. Every isomorphism between two such trees maps the center to the center, and hence is a rooted isomorphism after rooting both trees at their centers. This proves the second statement.
\end{proof}

\begin{corollary}\label{cor:signout}
    Let $S$ be a rooted oriented tree, with a labeling induced by its orientation. Then
    \[
        \sgn_S(\alpha)=(-1)^r
    \]
    for every rooted isomorphism $\alpha:S\to S_\out$, where $r$ is the number of edges of $S$ directed towards the root. In particular, if $S_1$ and $S_2$ are isomorphic rooted oriented trees and $\alpha:S_1\to S_\out$, $\beta:S_2\to S_\out$ are rooted isomorphisms, then $\sgn_{S_1}(\alpha)=\sgn_{S_2}(\beta)$ if the corresponding numbers of inward edges have the same parity, and the two signs are opposite otherwise.
\end{corollary}

\begin{proof}
    Consider first one rooted isomorphism $\alpha:S\to S_\out$. Since the labels of both trees are induced by their orientations, an edge contributes $-1$ precisely when its direction in $S$ is opposite to its direction in $S_\out$, i.e., precisely when it is directed towards the root. Thus the product is $(-1)^r$. Lemma \ref{lem:constantsigneventree} shows that the value does not depend on the chosen rooted isomorphism. The final assertion follows by applying this formula to $S_1$ and $S_2$.
\end{proof}

\subsection{Trees of diameter at most $3$}\label{subsec:trees3}

In this subsection, we consider trees of diameter two and three. In particular, we give another proof of Theorem \ref{thm:ai}.

\begin{proposition}\label{prop:stars}
    A star is converse invariant if and only if it is self-converse or a path.
\end{proposition}

\begin{proof}
    The path case follows from Theorem \ref{thm:paths-cycles}. Let $D$ be an orientation of a star that is not a path, let $u$ be its center, and let $a$ and $b$ be the indegree and outdegree of $u$, respectively. Thus $a+b\geq 3$.

    First suppose that $a+b$ is odd. Then $D$ has an odd number of edges. For the monomial $x^D$, formula \eqref{eq:coef-general} gives
    \[
        c_D=2\sum_{\alpha\in\Aut(D)}\sgn_D(\alpha).
    \]
    By Lemma \ref{lem:constantsignstar}, all terms in this sum have the same sign, so $c_D\neq 0$. Hence a converse invariant star must have $a+b$ even.

    Now assume $a+b$ is even. If necessary, replace $D$ by its converse, so that $b>0$. Let $H$ be the star obtained from $D$ by deleting one out-edge of the center. The copies of $H$ in $\UG(D)$ are obtained by deleting one leaf-edge of $D$. For each fixed copy, Lemma \ref{lem:constantsignstar} makes the inner signed sum constant. Comparing with the identity embedding shows that deleting an in-edge and deleting an out-edge contribute opposite signs. Consequently $c_H$ is a nonzero constant multiple of $a-b$. Since $D$ is converse invariant, $c_H=0$, and therefore $a=b$. This is exactly the condition that the oriented star be self-converse.
\end{proof}

We now give a polynomial proof of Theorem \ref{thm:ai}.

\begin{proof}[Proof of Theorem \ref{thm:ai}]
    By Proposition \ref{prop:stars}, it remains to consider orientations of double stars. Let $u$ and $v$ be the two non-leaf vertices and orient the central edge as $u\to v$. Let $a$ and $b$ be the numbers of in-leaves and out-leaves at $u$, respectively, and let $d$ and $c$ be the numbers of in-leaves and out-leaves at $v$, respectively. Put $p=a+b$ and $q=c+d$.

    We first show that $p=q$. Suppose $p\neq q$. If $p+q$ is odd, let $H$ be the forest obtained from $D$ by deleting the central edge $uv$. Then $H$ has an odd number of edges, and, because the two components have different sizes, it has a unique copy in $\UG(D)$. The signed automorphism sum is nonzero by Lemma \ref{lem:constantsignstar} applied to the two star components, so $c_H\neq 0$. If $p+q$ is even, then $D$ has an odd number of edges. Since the two inner vertices have different degrees, every automorphism fixes them, and the same constant-sign argument gives $c_D\neq0$. In either case $D$ is not converse invariant, a contradiction. Hence $p=q$.

    If $p=1$, the underlying tree is a path, which is excluded. Assume first that $p\geq 3$. After possibly replacing $D$ by its converse and swapping the names of $u$ and $v$, we may suppose that $a>0$. Let $H$ be obtained from $D$ by deleting the central edge and one in-edge incident with $u$. The copies of $H$ in $\UG(D)$ are obtained by deleting the central edge and one arbitrary leaf-edge. Since $p\geq3$, the two components of $H$ are stars of distinct sizes at least $2$, and Lemma \ref{lem:constantsignstar} makes each inner automorphism sum constant. A direct comparison with the identity embedding gives
    \[
        c_H=C(a-b-c+d)
    \]
    for some nonzero constant $C$ depending only on $p$ and on the labels. Since $c_H=0$ and $a+b=c+d$, we obtain $a=c$ and $b=d$. Thus $D$ is self-converse.

    It remains to handle $p=2$. Up to taking the converse and interchanging $u$ and $v$, the double stars with two leaves at each inner vertex are either self-converse, the exceptional bridge-mirrored tree in Figure \ref{fig:bmex}, or the case
    \[
        (a,b,c,d)=(2,0,1,1).
    \]
    In this last case, let $H$ be the oriented star obtained by deleting from $D$ the two leaf-edges incident with $v$. The only copies of $\UG(H)$ in $\UG(D)$ are the stars centered at $u$ and at $v$. For both copies, the signed automorphism sum has the same sign by Lemma \ref{lem:constantsignstar} (equivalently, by Corollary \ref{cor:signout}, the two stars have the same parity of inward edges). Hence $c_H\neq0$, so this remaining case is not converse invariant. Therefore the only converse invariant double stars are the self-converse ones and the exceptional bridge-mirrored tree.
\end{proof}

\subsection{Trees of diameter $4$}\label{subsec:trees4}

In this subsection, $T$ will denote a converse invariant oriented tree of diameter four, with central vertex $u$.  We first prove a parity lemma valid for all such trees.  We then prove the classification (Theorem \ref{thm:diam4classification}) in two parts, according to whether $u$ has at least four non-leaf neighbors. 

Let $S_t$ ($t \geq 0$) denote the set of neighbors of $u$ with degree exactly $t+1$ (i.e. that have $t$ children leaves), and $T_t$ denote the subtree of $T$ induced by $u$, $S_t$ and the leaves connected to the vertices of $S_t$.  For $0 \leq l \leq t$, denote by $A_l^t$ the set of out-neighbors of $u$ in $S_t$ with exactly $l$ leaves pointing towards them, and $a_l^t=|A_l^t|$.  Symmetrically, let $B_l^t$ denote the set of in-neighbors of $u$ in $S_t$ with exactly $l$ leaves pointing from them, and $b_l^t=|B_l^t|$.  In particular, $T_t$ being self-converse is equivalent to $a_l^t=b_l^t$ for every $l \in \{0,\dots,t\}$, and $T$ is self-converse if and only if this condition holds for every $t$.

We write
\[
        p(T)=\sum_{t\geq 1}|S_t|.
\]
Thus $p(T)$ is the number of neighbors of the central vertex which are not leaves. 

\begin{lemma}\label{lem:evennumberofedges}
Let $T$ be a converse invariant orientation of a tree of diameter four,
with center $u$. Then $|S_t|$ is even for every $t\geq 0$. In particular,
$T$ has an even number of edges.
\end{lemma}

\begin{proof}
First we prove that $T$ has an even number of edges. Suppose, to the
contrary, that $|E(T)|$ is odd. Since $\UG(T)$ has diameter four, it has
the unique center $u$. Hence every automorphism of $\UG(T)$ fixes $u$.
By Lemma~\ref{lem:constantsigneventree}, the quantity $\sgn_T(\alpha)$
is the same for every automorphism $\alpha$ of $\UG(T)$. Therefore the
coefficient of the monomial $x^T$ is
$$
        c_T
        =
        2\sum_{\alpha\in\Aut(\UG(T))}\sgn_T(\alpha),
$$
which is non-zero. This contradicts the converse invariance of $T$. Thus
$|E(T)|$ is even.

We shall use the following simple rooted-profile observation. Suppose that
a copy of a connected subgraph of $\UG(T)$ is rooted at $u$. Its rooted
branch profile records, for each $s\geq0$, the number of neighbors of
$u$ which have exactly $s$ leaf-neighbors inside the copy. If one deletes
a branch in $S_0$, then this profile changes only by decreasing the
number of $0$-branches by one. If one deletes one leaf-edge from a branch
in $S_s$, with $s\geq1$, and ignores the isolated leaf, then the profile
changes by decreasing the number of $s$-branches by one and increasing
the number of $(s-1)$-branches by one. Finally, if one deletes the central
edge of a branch in $S_s$, then one creates a detached star $K_{1,s}$.
These effects are distinct and will be used below.

We start with $t=0$. Suppose first that $S_0\neq\emptyset$. Let $H$ be the graph obtained from
$T_{\out}$ by deleting one edge from $u$ to a vertex of $S_0$, ignoring
the isolated vertex. Since deleting such an edge does not change the
positive branches, the component of $H$ containing $u$ still has diameter
four. Therefore every copy of $\UG(H)$ in $\UG(T)$ is rooted at the
original center $u$.

The rooted branch profile of $H$ has one fewer $0$-branch than $T$ and
the same number of $s$-branches for every $s\geq1$. By the rooted-profile
observation above, the only way to obtain this profile is to delete one
branch in $S_0$. Hence every copy of $\UG(H)$ is obtained by deleting an
edge from $u$ to a vertex of $S_0$.

For a fixed copy, Lemma~\ref{lem:constantsigneventree} shows that the
inner sum over isomorphisms is a non-zero constant. Moreover, deleting an
edge directed away from $u$ gives the opposite sign from deleting an edge
directed towards $u$. Hence
$$
        c_H=C(a_0^0-b_0^0)
$$
for some non-zero constant $C$. Since $T$ is converse invariant, we have
$c_H=0$, and therefore $a_0^0=b_0^0$. Thus
$$
        |S_0|=a_0^0+b_0^0
$$
is even. If $S_0=\emptyset$, this conclusion is trivial.

It remains to prove that $|S_t|$ is even when $t \geq 1$. Since $\UG(T)$ has diameter four, we have $p(T)\geq2$. We first prove the result when $p(T)\geq3$.

Suppose that $S_1\neq\emptyset$. Let $H$ be obtained from $T_{\out}$
by deleting the unique leaf-edge incident with one vertex of $S_1$,
again ignoring the isolated leaf. Since $p(T)\geq3$, the component
containing $u$ still has at least two positive branches, and so has
diameter four. Thus every copy is rooted at $u$.

The rooted branch profile of $H$ is obtained from that of $T$ by replacing
one $1$-branch by one $0$-branch. By the rooted-profile observation above,
this can only come from deleting the unique leaf-edge of a branch in
$S_1$. Therefore every copy of $\UG(H)$ is obtained by shortening one
branch of $S_1$ to a branch of $S_0$.

For a fixed copy, the inner isomorphism sum is again a non-zero constant.
The sign of the deleted leaf-edge is $+1$ or $-1$ according to whether
that leaf-edge is directed away from or towards the vertex of $S_1$.
Consequently $c_H$ is a signed sum of $|S_1|$ signs, all with the same
non-zero multiplicative constant. If $|S_1|$ were odd, this signed sum
could not be zero. Since $c_H=0$, we conclude that $|S_1|$ is even. If
$S_1=\emptyset$, this is trivial.

Now suppose that $t\geq2$. If $S_t=\emptyset$, there is nothing to
prove. Let $H$ be obtained from $T_{\out}$ by deleting the central edge
$uv$ for one vertex $v\in S_t$, while keeping all leaf-edges incident
with $v$. Then $H$ has no isolated vertices and has $|V(T)|$ vertices and
$|E(T)|-1$ edges. Hence every copy of $\UG(H)$ in $\UG(T)$ is obtained
from $\UG(T)$ by deleting exactly one edge.

The deleted edge cannot be a leaf-edge, since that would create an
isolated vertex. It must therefore be a central edge. Moreover, deleting
the central edge of a branch in $S_s$ leaves a detached star $K_{1,s}$.
Thus, in order to obtain a copy of $\UG(H)$, we must have $s=t$. Hence
the copies of $\UG(H)$ are precisely the graphs obtained by deleting the
central edge of one branch in $S_t$.

Since $p(T)\geq3$, after deleting one positive branch the component
containing $u$ still has at least two positive branches, and therefore
has diameter four. Thus this component is rooted at $u$. The other
component is a star $K_{1,t}$. By Lemma~\ref{lem:constantsigneventree}
and Lemma~\ref{lem:constantsignstar}, the inner isomorphism sum for each
copy is a non-zero constant. Deleting a central edge oriented away from
$u$ gives the opposite sign from deleting a central edge oriented towards
$u$. Hence $c_H$ is a signed sum of $|S_t|$ signs, all with the same
non-zero multiplicative constant. If $|S_t|$ were odd, then $c_H\neq0$,
contradicting converse invariance. Therefore $|S_t|$ is even.

This proves the lemma when $p(T)\geq3$.

Finally, let us consider the case $p(T)=2$. Let $x$ and $y$ be the two
non-leaf neighbors of $u$. Suppose that $x$ has $r$ leaf-neighbors and
$y$ has $s$ leaf-neighbors, with $1\leq r\leq s$, and $m = |S_0|$.

We now prove that $r=s$. Suppose, for a contradiction, that $r<s$. Since
$|E(T)|$ is even and $m$ is even, the number $r+s$ is even.

First assume that $m\neq r$. Let $H$ be obtained from $T_{\out}$ by
deleting the central edge $uy$ of the larger branch. Then $H$ has two
components: a detached star $K_{1,s}$, and the component $R$ containing
$u$, which consists of the branch at $x$ together with the $m$ leaves
adjacent to $u$.

If $m=0$, then $R$ is a star $K_{1,r+1}$. Since $r+s$ is even, we cannot
have $s=r+1$, and hence the two components of $H$ are not isomorphic. If
$m>0$, then $R$ has diameter three, whereas $K_{1,s}$ has diameter two,
so again the two components are not isomorphic.

Every copy of $\UG(H)$ in $\UG(T)$ has all vertices of $T$ and one edge
less than $T$. Thus it is obtained by deleting one edge of $\UG(T)$. A
leaf-edge deletion or an $S_0$-edge deletion would create an isolated
vertex, which $H$ does not have. Hence the deleted edge must be one of
$ux$ and $uy$. Deleting $uy$ gives the graph $H$. Deleting $ux$ gives a
detached star $K_{1,r}$ and the component consisting of the branch at
$y$ together with the $m$ leaves adjacent to $u$, which is not isomorphic
to $H$ because $r<s$ and the component types described above do not
match. Therefore every copy of $\UG(H)$ is obtained by deleting $uy$.

For this unique type of copy, the inner isomorphism sum is non-zero. If
$m=0$, both components are non-isomorphic stars, and this follows from
Lemma~\ref{lem:constantsignstar}. If $m>0$, the component $R$ has two
non-leaf vertices of degrees $m+1$ and $r+1$, which are distinct because
$m\neq r$; hence every automorphism of $R$ fixes these two vertices and
only permutes leaves, so the signed automorphism sum is again non-zero.
Thus $c_H\neq0$, contradicting converse invariance.

We may therefore assume that $m=r$. Since $m$ is even, this gives
$r\geq2$. Now let $H$ be obtained from $T_{\out}$ by deleting the central
edge $ux$ of the smaller branch. The detached component is the star
$K_{1,r}$, which is not a single edge. The component containing $u$
consists of the branch at $y$ together with the $m=r$ leaves adjacent to
$u$. Its two non-leaf vertices have degrees $s+1$ and $r+1$, which are
distinct because $s>r$.

As before, every copy of $\UG(H)$ must be obtained by deleting one of the
two central edges $ux$ or $uy$. Deleting $ux$ gives $H$. Deleting $uy$
gives a detached star $K_{1,s}$ and a component whose two non-leaf
vertices both have degree $r+1$, and this is not isomorphic to $H$.
Therefore there is only one type of copy of $\UG(H)$ in $\UG(T)$.

The inner isomorphism sum is non-zero: the detached component is a star
with at least three vertices, and the other component has two non-leaf
vertices of distinct degrees, so its automorphisms only permute leaves.
Consequently $c_H\neq0$, again contradicting converse invariance.

This contradiction proves that $r=s$. Therefore the two positive branches
both lie in the same set $S_r$, so $|S_r|=2$, while every other $S_t$ with
$t\geq1$ is empty. We have already proved that $|S_0|$ is even. Hence
$|S_t|$ is even for every $t\geq0$ in the case $p(T)=2$ as well.

Combining the cases $p(T)\geq3$ and $p(T)=2$ proves the lemma.
\end{proof}

We now begin to show Theorem \ref{thm:diam4classification} in the case $p(T)\geq4$.  In this case we prove that $T$ is self-converse by showing inductively that each $T_t$ is self-converse. We shall use the following polynomials: for $0\leq r\leq t$, put
\[
        K_r^{(t)}(x)=\sum_{q=0}^r(-1)^q\binom{x}{q}\binom{t-x}{r-q}.
\]
Also put
\[
        M_r^t=\sum_{k=0}^t(a_k^t-b_k^t)K_r^{(t)}(k),
        \qquad
        N_r^t=\sum_{k=0}^t(a_k^t+b_k^t)K_r^{(t)}(k).
\]
The goal is to prove that $M_r^t=0$ for every $0\leq r\leq t$ and every $t\geq0$.

\begin{lemma}\label{lem:diam4local}
Let $w\in S_s$.
\begin{enumerate}[label=\textup{(\roman*)}]
    \item Suppose that the edge $uw$ is kept and exactly $r$ leaf-edges incident with $w$ are deleted.  Then the signed sum over all choices of these $r$ deleted leaf-edges is
    \[
        K_r^{(s)}(k)
    \]
    if $w\in A_k^s$, and
    \[
        (-1)^rK_r^{(s)}(k)
    \]
    if $w\in B_k^s$.

    \item Suppose that the edge $uw$ is deleted and all leaf-edges incident with $w$ are kept.  Then the corresponding factor is
    \[
        \eta(w)=
        \begin{cases}
        +1, & w\in A_k^s,\\
        -1, & w\in B_k^s.
        \end{cases}
    \]

    \item If $s=1$ and the whole branch at $w$ is deleted, then the corresponding factor is $K_1^{(1)}(k)$ both for $w\in A_k^1$ and for $w\in B_k^1$.
\end{enumerate}
\end{lemma}

\begin{proof}
Suppose first that $w\in A_k^s$.  Then $u\to w$, and exactly $k$ leaf-edges point towards $w$.  If $r$ leaf-edges are deleted and exactly $q$ of them are among these $k$ inward leaf-edges, then the sign changes by $(-1)^q$.  Summing over $q$ gives
\[
        \sum_{q=0}^r(-1)^q\binom{k}{q}\binom{s-k}{r-q}=K_r^{(s)}(k).
\]

Now suppose that $w\in B_k^s$.  Then $w\to u$, and exactly $k$ leaf-edges point away from $w$, so $s-k$ leaf-edges point towards $w$.  Deleting $r$ leaf-edges gives the leaf contribution
\[
        \sum_{q=0}^r(-1)^q\binom{s-k}{q}\binom{k}{r-q}=(-1)^rK_r^{(s)}(k).
\]
Since the central edge is kept in both the original branch and the selected sub-branch, the central sign cancels in the ratio.  This proves \textup{(i)}.

If the central edge is deleted and all leaves are kept, then the leaf contribution is unchanged, and only the sign of the central edge is removed.  This gives $+1$ in the $A$-case and $-1$ in the $B$-case, proving \textup{(ii)}.

Finally, assume $s=1$ and the whole branch is deleted.  The selected contribution is $1$, so the ratio is the sign of the whole branch.  In both the $A_k^1$-case and the $B_k^1$-case this sign is $(-1)^k=K_1^{(1)}(k)$.  This proves \textup{(iii)}.
\end{proof}

\begin{definition}\label{def:diam4testgraphs}
Assume that the relevant set $S_t$ is non-empty.
\begin{enumerate}[label=\textup{(\roman*)}]
    \item For $t=0$ or $t\geq2$, let $Z^t$ be obtained from $T_\out$ by deleting the central edge $uv$ for one vertex $v\in S_t$, keeping all leaf-edges incident with $v$ if $t\geq2$.

    \item For $t\geq1$ and $l$ odd, $1\leq l\leq t$, let $P_l^t$ be obtained from $T_\out$ by deleting $l$ leaf-edges incident with one vertex $v\in S_t$, while keeping the central edge $uv$.

    \item For $t\geq2$ and $l\geq2$ even, $l\leq t$, let $Q_l^t$ be obtained from $T_\out$ by choosing two distinct vertices $a,v\in S_t$, deleting the central edge $ua$, keeping all leaf-edges incident with $a$, and deleting $l$ leaf-edges incident with $v$, while keeping $uv$.

    \item If $S_1\neq\emptyset$, let $Y^1$ be obtained from $T_\out$ by choosing two distinct vertices $a,v\in S_1$, deleting the whole branch at $a$, and deleting the unique leaf-edge incident with $v$.
\end{enumerate}
\end{definition}

The two distinct vertices required in the definitions of $Q_l^t$ and $Y^1$ exist by Lemma \ref{lem:evennumberofedges}.  Moreover, all the graphs in Definition \ref{def:diam4testgraphs} have an odd number of edges, since $T$ has an even number of edges and in each case we delete an odd number of edges.

\begin{lemma}\label{lem:diam4anchor}
Assume that $p(T)\geq4$.  For every graph in Definition \ref{def:diam4testgraphs}, every copy of its underlying graph in $\UG(T)$ sends the component of diameter four to a subgraph centered at the original center $u$ of $T$.  Consequently, every such copy may be treated as a rooted copy with root $u$ in the coefficient formula.  Moreover, none of these graphs has a connected component consisting of a single edge.
\end{lemma}

\begin{proof}
The graph $P_l^t$ is connected.  It is obtained by shortening at most one positive branch, so it still contains at least $p(T)-1\geq3$ positive branches.  Hence it contains a path of length four, and has diameter four.

For $Z^t$ with $t\geq2$, the graph has one detached star $K_{1,t}$ and one component containing $u$.  The component containing $u$ has at least $p(T)-1\geq3$ positive branches, and hence has diameter four.  The detached component has diameter two.  For $Z^0$, the graph is connected and still contains all positive branches, so it has diameter four.

For $Q_l^t$, the component containing $u$ loses the anchor branch, and the probe branch may be shortened all the way to a leaf-branch.  Nevertheless it still contains at least $p(T)-2\geq2$ positive branches, and hence has diameter four.  The detached component is a star $K_{1,t}$ with $t\geq2$.  The same argument applies to $Y^1$: after deleting one $S_1$-branch and shortening another one, at least two positive branches remain, so the remaining component has diameter four.

Every path of length four in $\UG(T)$ has middle vertex $u$, because $u$ is the center of the diameter-four tree $\UG(T)$.  Hence any copy of a component of diameter four is necessarily rooted at $u$.  This proves the anchoring assertion.  The final assertion follows directly from the definitions: the only detached components are stars $K_{1,t}$ with $t\geq2$.
\end{proof}

Because of Lemmas \ref{lem:diam4anchor}, \ref{lem:constantsigneventree}, and \ref{lem:constantsignstar}, the inner sum over isomorphisms in the coefficient formula contributes only a non-zero constant depending on the test graph and on the rooted profile of the copy.  Thus, from now on, we compute coefficients up to non-zero multiplicative constants.  This is harmless, since we only use the fact that these coefficients vanish.

\paragraph{Branches and sub-branches.}
For a vertex $w\in S_s$, let $L(w)$ be the set of the $s$ leaves adjacent with $w$.  The \emph{branch} at $w$ is the subgraph
\[
        \mathcal B(w)=T[\{u,w\}\cup L(w)].
\]
Thus $E(\mathcal B(w))=\{uw\}\cup\{wx:x\in L(w)\}$, and we call $s$ the \emph{size} of the branch.  If $F$ is a subgraph of $\UG(T)$, the \emph{trace} of $F$ on the branch $\mathcal B(w)$ is
\[
        \operatorname{tr}_F(w)=E(F)\cap E(\mathcal B(w)).
\]
When the copy is rooted at $u$, each non-empty trace has one of the following forms.
\begin{itemize}
    \item A \emph{central $r$-sub-branch}, denoted $C_r(w)$, consists of the central edge $uw$ together with exactly $r$ leaf-edges incident with $w$, where $0\leq r\leq s$.
    \item A \emph{free $r$-sub-branch}, denoted $F_r(w)$, consists of exactly $r$ leaf-edges incident with $w$ and does not contain $uw$, where $1\leq r\leq s$.  This is a detached star $K_{1,r}$.
    \item The trace may also be empty; in that case the whole branch is absent from the chosen subgraph.
\end{itemize}
If $w\in S_s$ and the trace is $C_r(w)$ with $r<s$, we say that the branch has been \emph{shortened from $s$ to $r$}, and we write $s\to r$.  If the trace is a detached free star or is empty, we record this by an edge $s\to\ast$ in the bookkeeping below.  This notation records only the type of the copy; the actual choices of the vertices and of the retained or deleted leaves are counted in the coefficient calculation.

\begin{lemma}\label{lem:diam4profiles}
Assume $p(T)\geq4$.
\begin{enumerate}[label=\textup{(\roman*)}]
    \item Every rooted copy of $Z^t$, for $t=0$ or $t\geq2$, is principal: it is obtained by deleting the central edge of one branch in $S_t$, and leaving all other branches unchanged.  For $t\geq2$ the deleted central edge leaves a detached free $t$-sub-branch; for $t=0$ the deleted branch has empty trace.

    \item Let $t\geq1$ and let $l$ be odd, $1\leq l\leq t$.  Every rooted copy of $P_l^t$ is described by a strictly decreasing chain
    \[
        t=t_0>t_1>\cdots>t_j=t-l.
    \]
    If $l_i=t_i-t_{i+1}$, then the copy is obtained by shortening one branch in $S_{t_i}$ from $t_i$ to $t_{i+1}$, for each $i=0,\dots,j-1$, and leaving every other branch unchanged.

    \item Let $t\geq2$ and let $l\geq2$ be even.  Every rooted copy of $Q_l^t$ is obtained as follows.  First choose an anchor branch in $S_t$ and replace it by a detached free $t$-sub-branch, that is, delete its central edge and keep all its leaves.  Then, in the central component, choose a chain
    \[
        t=t_0>t_1>\cdots>t_j=t-l
    \]
    as in part \textup{(ii)}.  The anchor branch is distinct from the branch of $S_t$ used at the first step of the chain.

    \item Every rooted copy of $Y^1$ has one of the following two forms.  Either one branch in $S_1$ has empty trace and a distinct branch in $S_1$ is shortened from $1$ to $0$, or one branch in $S_0$ has empty trace and two distinct branches in $S_1$ are shortened from $1$ to $0$.
\end{enumerate}
\end{lemma}

\begin{proof}
By Lemma \ref{lem:diam4anchor}, every copy of the test graphs is rooted at the original center $u$.  Therefore a copy can be analyzed branch by branch using the traces defined above.  We shall repeatedly use the following elementary bookkeeping fact.

Let $m_s=|S_s|$.  Suppose that, in a rooted copy, a branch originally in $S_s$ is shortened to a central $r$-sub-branch, with $r<s$.  We record this by a directed edge $s\to r$.  Suppose also that a branch originally in $S_s$ does not appear in the central component, either because it becomes a detached free star or because it has empty trace.  We record this by a directed edge $s\to\ast$.  Let $m'_h$ be the number of central $h$-sub-branches in the test graph.  Then, for every level $h$, we have
\[
        \operatorname{out}(h)-\operatorname{in}(h)=m_h-m'_h,        \tag{$\ast$}
\]
where $\operatorname{out}(h)$ counts both ordinary edges $h\to r$ and edges $h\to\ast$, while $\operatorname{in}(h)$ counts only ordinary edges $r\to h$.  All ordinary directed edges go strictly down.  Hence the finite directed multigraph formed by these edges is acyclic, and its non-trivial part decomposes into directed paths starting at the levels where $m_h-m'_h>0$ and ending either at levels where $m_h-m'_h<0$ or at the sink $\ast$.

For $Z^t$ with $t\geq2$, the test graph has one detached star $K_{1,t}$ and, in the central component, the same branch profile as $T$ except that one central $t$-branch is missing.  Thus $m_t-m'_t=1$, all other differences are zero, and there is exactly one edge ending at $\ast$.  By the path decomposition above, the only possible source is level $t$.  The path ending at $\ast$ must therefore start at $t$.  It cannot contain an ordinary first step $t\to r<t$, because a branch of size $r<t$ cannot contain a free $t$-sub-branch.  Hence this path is the single edge $t\to\ast$, and the copy is principal.  The case $t=0$ is the same, except that the missing branch has empty trace; since there is no ordinary edge out of level $0$, all other branches are unchanged.

For $P_l^t$, the graph is connected and contains every central edge.  Therefore there are no edges of the form $s\to\ast$.  Its central profile is obtained from that of $T$ by replacing one central $t$-branch by one central $(t-l)$-branch.  Hence the directed multigraph has one unit of source at $t$ and one unit of sink at $t-l$.  Since all ordinary edges strictly decrease the level, its non-trivial part is a single directed path
\[
        t=t_0\to t_1\to\cdots\to t_j=t-l.
\]
Writing $l_i=t_i-t_{i+1}$, the edge $t_i\to t_{i+1}$ says exactly that one branch in $S_{t_i}$ is shortened by deleting $l_i$ of its leaf-edges.  This proves part \textup{(ii)}.

For $Q_l^t$, the graph has one detached star $K_{1,t}$, coming from the anchor, and its central component is obtained from the profile of $T$ by deleting one central $t$-branch and replacing another central $t$-branch by a central $(t-l)$-branch.  Thus the directed multigraph has two paths starting at level $t$: one ending at $\ast$, and one ending at $t-l$.  As in the case of $Z^t$, the path ending at $\ast$ must be the single edge $t\to\ast$.  This is the anchor branch.  The other path is an ordinary decreasing path
\[
        t=t_0\to t_1\to\cdots\to t_j=t-l,
\]
which gives the chain of shortenings in the central component.  These two paths start with two different outgoing edges from level $t$, so they correspond to two distinct branches of $S_t$.  This proves part \textup{(iii)}.

Finally, consider $Y^1$.  The graph $Y^1$ has no detached non-trivial component, but exactly one branch has empty trace.  Its central profile is obtained from the profile of $T$ by removing two central $1$-branches and adding one central $0$-branch.  Equivalently,
\[
        m_h-m'_h=
        \begin{cases}
        2, & h=1,\\
        -1, & h=0,\\
        0, & h\geq2.
        \end{cases}
\]
Thus the non-trivial directed multigraph consists of two paths which start at level $1$: one path ends at level $0$, and the other ends at $\ast$.  Since the only level below $1$ is $0$, there are exactly two possibilities.  Either the path to $\ast$ is the single edge $1\to\ast$, and the other path is $1\to0$; or the path to $\ast$ is $1\to0\to\ast$, and the other path is another edge $1\to0$.  These are precisely the two cases described in the statement.
\end{proof}

\begin{lemma}\label{lem:coefficientshtl}
Assume that $p(T)\geq4$, and suppose that $T_s$ is self-converse for every $s<t$.
\begin{enumerate}[label=\textup{(\roman*)}]
    \item If $t=0$ or $t\geq2$, then
    \[
        c_{Z^t}=\lambda_t M_0^t
    \]
    for some non-zero constant $\lambda_t$.

    \item If $t\geq1$ and $l$ is odd, $1\leq l\leq t$, then
    \[
        c_{P_l^t}=\lambda_{t,l}M_l^t+
        \sum_{\substack{0\leq r<l\\ r\text{ odd}}}\lambda_{t,l,r}M_r^t,
    \]
    where $\lambda_{t,l}\neq0$.

    \item If $t\geq2$ and $l\geq2$ is even, then, after using $M_0^t=0$,
    \[
        c_{Q_l^t}=\mu_{t,l}M_l^t+
        \sum_{\substack{0<r<l\\ r\text{ even}}}\mu_{t,l,r}M_r^t,
    \]
    where $\mu_{t,l}\neq0$.

    \item If $S_1\neq\emptyset$, then
    \[
        c_{P_1^1}=\lambda M_1^1
    \]
    with $\lambda\neq0$, and using $M_0^0=0$,
    \[
        c_{Y^1}=\mu\bigl(N_1^1M_1^1-M_0^1\bigr)
    \]
    for some non-zero constant $\mu$.
\end{enumerate}
\end{lemma}

\begin{proof}
By Lemma \ref{lem:diam4anchor}, every copy is rooted.  As observed above, the sum over isomorphisms gives only a non-zero multiplicative constant. The constants may depend on the rooted profile, but this does not affect the triangular argument: every non-principal surviving profile contributes
only a previously known moment, while the principal profile is the unique profile contributing the leading moment.  We therefore compute only the signed sum of rooted copies.

For $Z^t$, every copy is principal by Lemma \ref{lem:diam4profiles}.  The unique modified branch contributes $+1$ if it lies in some $A_k^t$, and $-1$ if it lies in some $B_k^t$.  Therefore the signed sum is a non-zero constant times
\[
        \sum_{k=0}^t(a_k^t-b_k^t)=M_0^t.
\]
This proves \textup{(i)}.

Let $l$ be odd and consider $P_l^t$.  By Lemma \ref{lem:diam4profiles}, every rooted copy is encoded by a chain
\[
        t=t_0>t_1>\cdots>t_j=t-l.
\]
Put $l_i=t_i-t_{i+1}$.  By Lemma \ref{lem:diam4local}, the local contribution of the step $t_i\to t_{i+1}$ is $M_{l_i}^{t_i}$ if $l_i$ is odd, and $N_{l_i}^{t_i}$ if $l_i$ is even.  If some lower step $l_i$, $i\geq1$, is odd, then the factor $M_{l_i}^{t_i}$ vanishes by the induction hypothesis, because $t_i<t$ and $T_{t_i}$ is self-converse.  Hence every surviving chain has all lower steps even.  Since $l=l_0+\cdots+l_{j-1}$ is odd, the first step $l_0$ is odd.  The principal chain has $l_0=l$ and gives a non-zero multiple of $M_l^t$.  Every other surviving chain has $l_0<l$ and gives a multiple of a previous odd moment $M_{l_0}^t$.  This proves \textup{(ii)}.

Now let $l\geq2$ be even and consider $Q_l^t$.  By Lemma \ref{lem:diam4profiles}, a rooted copy consists of an anchor in $S_t$ and a decreasing chain
\[
        t=t_0>t_1>\cdots>t_j=t-l
\]
in the central component.  As in the previous paragraph, any lower odd step gives a vanishing factor by the induction hypothesis.  Thus every surviving chain has all lower steps even, and so $l_0$ is even.  The lower steps contribute only constants of type $N_{l_i}^{t_i}$.

It remains to compute the combined contribution of the anchor and the first step of the chain.  For a fixed even $r$, the anchor contributes
\[
        \eta(a)=
        \begin{cases}
        +1, & a\in A_k^t,\\
        -1, & a\in B_k^t,
        \end{cases}
\]
while the first step contributes $K_r^{(t)}(k(v))$ with the same sign for $A$- and $B$-vertices.  Since the anchor and the first vertex of the chain are distinct, the same-level sum is
\[
        \sum_{\substack{a,v\in S_t\\a\neq v}}
        \eta(a)K_r^{(t)}(k(v))
        =
        \left(\sum_{a\in S_t}\eta(a)\right)
        \left(\sum_{v\in S_t}K_r^{(t)}(k(v))\right)
        -
        \sum_{v\in S_t}\eta(v)K_r^{(t)}(k(v)).
\]
The first factor is $M_0^t$, and the last sum is $M_r^t$.  After using $M_0^t=0$, this is $-M_r^t$.  The principal chain has $r=l$ and gives a non-zero multiple of $M_l^t$; every other surviving chain has even $r<l$ and gives a previous even moment.  This proves \textup{(iii)}.

It remains to handle $t=1$.  For $P_1^1$, the profile lemma gives only the principal profile $1\to0$, and Lemma \ref{lem:diam4local} gives a non-zero multiple of $M_1^1$.

For $Y^1$, first consider the profiles in which one $S_1$-branch is deleted completely and another $S_1$-branch is shortened from $1$ to $0$.  By Lemma \ref{lem:diam4local}, the deleted branch contributes $K_1^{(1)}(k)$, regardless of whether it is of type $A$ or type $B$.  The shortened branch contributes $K_1^{(1)}(k)$ in the $A$-case and $-K_1^{(1)}(k)$ in the $B$-case.  Thus this part of the signed sum is a non-zero constant times
\[
        \sum_{\substack{a,v\in S_1\\a\neq v}}
        K_1^{(1)}(k(a))\eta(v)K_1^{(1)}(k(v)).
\]
Since $K_1^{(1)}(0)=1$ and $K_1^{(1)}(1)=-1$, we have $(K_1^{(1)}(k))^2=1$.  Therefore the last display is
\[
        N_1^1M_1^1-M_0^1.
\]
The only remaining profiles are those in which one $S_0$-branch is deleted and two $S_1$-branches are shortened.  The deleted $S_0$-branch contributes $+1$ for $A_0^0$ and $-1$ for $B_0^0$, so the total contribution of all such profiles is a multiple of $M_0^0=0$.  This proves \textup{(iv)}.
\end{proof}

\begin{theorem}\label{thm:largecore}
    Let $T$ be an oriented tree of diameter four, with central vertex $u$, and suppose that
    \[
        \sum_{t\geq1}|S_t|\geq4.
    \]
    Then $T$ is converse invariant if, and only if, $T$ is self-converse.
\end{theorem}

\begin{proof}
If $T$ is self-converse, then it is clearly converse invariant.  Conversely, suppose that $T$ is converse invariant.  Then every coefficient of the polynomial $P_T$ is zero.  By Lemma \ref{lem:evennumberofedges}, $T$ has an even number of edges, and hence all the test graphs above have an odd number of edges.  Therefore their coefficients occur in $P_T$ and vanish.

We prove by induction on $t$ that $T_t$ is self-converse.  Equivalently, we prove that $a_k^t=b_k^t$ for every $0\leq k\leq t$.

First consider $t=0$.  If $S_0=\emptyset$, there is nothing to prove.  If $S_0\neq\emptyset$, then $c_{Z^0}=0$, and Lemma \ref{lem:coefficientshtl} gives
\[
        M_0^0=a_0^0-b_0^0=0.
\]
Thus $T_0$ is self-converse.

Now consider $t=1$.  If $S_1=\emptyset$, there is nothing to prove.  If $S_1\neq\emptyset$, then $c_{P_1^1}=0$, and Lemma \ref{lem:coefficientshtl} gives
\[
        M_1^1=0.
\]
Also $c_{Y^1}=0$.  Since $M_0^0=0$ has already been proved and $M_1^1=0$, Lemma \ref{lem:coefficientshtl} gives
\[
        M_0^1=0.
\]
Hence both moments $M_0^1$ and $M_1^1$ vanish.

Let now $t\geq2$, and suppose that $T_s$ is self-converse for every $s<t$.  If $S_t=\emptyset$, there is nothing to prove.  Assume $S_t\neq\emptyset$.  From $c_{Z^t}=0$ and Lemma \ref{lem:coefficientshtl}, we obtain
\[
        M_0^t=0.
\]
We next prove $M_l^t=0$ for $1\leq l\leq t$ by induction on $l$.  If $l$ is odd, then $c_{P_l^t}=0$ and Lemma \ref{lem:coefficientshtl} expresses $c_{P_l^t}$ as a non-zero multiple of $M_l^t$ plus moments $M_r^t$ with $r<l$, which are already known to vanish.  Hence $M_l^t=0$.  If $l$ is even, then $c_{Q_l^t}=0$ and Lemma \ref{lem:coefficientshtl}, using $M_0^t=0$, expresses $c_{Q_l^t}$ as a non-zero multiple of $M_l^t$ plus even moments $M_r^t$ with $0<r<l$, which are already known to vanish.  Hence again $M_l^t=0$.

Thus, for every $t$ with $S_t\neq\emptyset$, we have
\[
        M_l^t=0\qquad(0\leq l\leq t).
\]
It remains only to translate these moment equations into the equalities $a_k^t=b_k^t$.  The identity
\[
        K_l^{(t)}(x)=[z^l](1-z)^x(1+z)^{t-x}
\]
shows that $K_l^{(t)}(x)$ is a polynomial in $x$ of degree exactly $l$, with leading coefficient $(-2)^l/l!$.  Therefore $K_0^{(t)},K_1^{(t)},\dots,K_t^{(t)}$ form a basis of the vector space of polynomials of degree at most $t$.  Since
\[
        0=M_l^t=\sum_{k=0}^t(a_k^t-b_k^t)K_l^{(t)}(k)
        \qquad(0\leq l\leq t),
\]
we have
\[
        \sum_{k=0}^t(a_k^t-b_k^t)f(k)=0
\]
for every polynomial $f$ of degree at most $t$.  Taking $f$ to be the Lagrange polynomial which is $1$ at $j$ and $0$ at the other points of $\{0,1,\dots,t\}$ gives $a_j^t-b_j^t=0$.  Hence $a_j^t=b_j^t$ for every $j\in\{0,\dots,t\}$.

By induction, $T_t$ is self-converse for every $t\geq0$.  Hence $T$ itself is self-converse.
\end{proof}

It remains to treat the case $p(T)<4$.

\begin{lemma}\label{lem:remaining-shape}
Let $T$ be an orientation of a non-path tree of diameter four, with center
$u$.  Suppose that
$$
        \sum_{r\geq 1}|S_r|<4 .
$$
Then there is a unique integer $t\geq 1$ such that
$$
        S_t=\{x,y\},
        \qquad
        S_r=\emptyset \quad\text{for every }r\geq1,\ r\neq t .
$$
In particular, if $s=|S_0|$, then the underlying tree consists of the path
$xuy$, together with $t$ leaves adjacent to $x$, $t$ leaves adjacent to
$y$, and $s$ leaves adjacent to $u$.

\end{lemma}

\begin{proof}
Since the diameter is four, there are at least two neighbors of $u$ which
are not leaves.  Thus $\sum_{r\geq1}|S_r|\geq2$.  By Lemma
\ref{lem:evennumberofedges}, each $|S_r|$ is even.  Hence, under the
assumption $\sum_{r\geq1}|S_r|<4$, we must have
$$
        \sum_{r\geq1}|S_r|=2 .
$$
Again because every $|S_r|$ is even, these two vertices must lie in a single
set $S_t$, with $t\geq1$.  This proves the claim.
\end{proof}

\medskip

For the rest of this paragraph we assume the notation of Lemma
\ref{lem:remaining-shape}.  We write
$$
        X=N(x)\setminus\{u\},\qquad
        Y=N(y)\setminus\{u\},\qquad
        U=N(u)\setminus\{x,y\}.
$$
Thus $|X|=|Y|=t$ and $|U|=s$.  Let
$$
        o_x=|\{z\in X:x\to z\}|,
        \qquad
        o_y=|\{z\in Y:y\to z\}|,
        \qquad
        o_u=|\{z\in U:u\to z\}|.
$$
Finally, put
$$
        \epsilon_x=
        \begin{cases}
        1,& u\to x,\\
        -1,& x\to u,
        \end{cases}
        \qquad
        \epsilon_y=
        \begin{cases}
        1,& u\to y,\\
        -1,& y\to u.
        \end{cases}
$$

\begin{lemma}\label{lem:diam4-two-branch-equations}
Let $T$ be converse invariant and let $S_t=\{x,y\}$ be as above.  Then
\begin{equation}\label{eq:two-branch-basic}
        2o_u=s
        \qquad\text{and}\qquad
        o_x+o_y=t.
\end{equation}
Moreover, if $s\neq t$, then
\begin{equation}\label{eq:two-branch-central}
        \epsilon_x+\epsilon_y=0 .
\end{equation}
\end{lemma}

\begin{proof}
We use the coefficient formula from Section~\ref{sec:prelim}: if $H$ has an
odd number of edges, then
$$
        c_H=2(n-|V(H)|)!\sum_S\sum_\alpha
        \prod_{ij\in E(S)}y_{ij}(-1)^{(ij)_\alpha},
$$
and, since $T$ is converse invariant, every such coefficient is zero.

First suppose that $s>0$.  Let $H$ be the oriented graph obtained from
$T_\out$ by deleting one edge from $u$ to a vertex of $S_0$.  Every copy of
$\UG(H)$ in $\UG(T)$ is obtained by deleting one edge incident with a vertex
of $S_0$.  Indeed, deleting any other edge either disconnects one of the two
non-leaf branches or changes the number of non-leaf branches.  For each fixed
copy, all isomorphisms have the same sign by Lemma
\ref{lem:constantsigneventree}.  Hence, up to a non-zero constant independent of
which edge is deleted, the coefficient $c_H$ is
$$
        o_u-(s-o_u).
$$
Since $c_H=0$, we get $2o_u=s$.  If $s=0$, the same identity is trivial.
This proves the first identity in \eqref{eq:two-branch-basic}.

Next let $H$ be obtained from $T_\out$ by deleting one leaf-edge incident
with one of the two vertices in $S_t$.  A copy of $\UG(H)$ is obtained by
choosing one of the two branches, say the branch centered at $w\in\{x,y\}$,
and then deleting one of the $t$ leaf-edges incident with $w$.  The sign
change contributed by the deleted edge is positive if the deleted edge is
oriented away from $w$, and negative if it is oriented towards $w$.  Thus
the total contribution of the branch at $w$ is
$$
        o_w-(t-o_w)=2o_w-t.
$$
Therefore, up to a non-zero constant,
$$
        c_H=(2o_x-t)+(2o_y-t).
$$
Since $c_H=0$, this gives $o_x+o_y=t$, proving the second identity in \eqref{eq:two-branch-basic}.

It remains to prove \eqref{eq:two-branch-central}.  Assume $s\neq t$, and let
$H$ be obtained from $T_\out$ by deleting the edge from $u$ to one of the
vertices in $S_t$.  The graph $\UG(H)$ has one detached $t$-star and one
component containing the edge between $u$ and the other vertex of $S_t$.  In
that latter component the two non-leaf vertices have respectively $s$ and
$t$ leaf-neighbors, so they are distinguishable because $s\neq t$.  Hence
every copy is one of the two principal copies, obtained by deleting either
$ux$ or $uy$.  Again Lemma \ref{lem:constantsigneventree} shows that, for a
fixed copy, the inner sum over isomorphisms is a non-zero constant multiple of
the sign contributed by the deleted central edge.  This sign is $+1$ if the
deleted edge is oriented away from $u$, and $-1$ if it is oriented towards
$u$.  Consequently,
$$
        c_H=C(\epsilon_x+\epsilon_y)
$$
for some non-zero constant $C$.  Since $c_H=0$, we obtain
$\epsilon_x+\epsilon_y=0$.
\end{proof}

\begin{lemma}\label{lem:diam4-symmetric-two-branch}
Assume that $T$ is converse invariant, that $S_t=\{x,y\}$, and that
$s=t$.  If the two central edges have the same direction with respect to $u$, then $t$ is even
and
$$
        o_u=o_x=o_y=t/2 .
$$
\end{lemma}

\begin{proof}
It is enough to consider the case
$$
        u\to x,
        \qquad
        u\to y,
$$
since the other case is obtained by replacing $T$ with its converse.
By Lemma \ref{lem:diam4-two-branch-equations},
$$
        o_u=t/2
        \qquad\text{and}\qquad
        o_x+o_y=t.
$$
In particular, $t$ is even.  Write $t=2h$, so that $o_u=h$.  We shall
prove that $o_x=o_y=h$.

Let $H$ be an outward oriented star with three edges.  Since $T$ is
converse invariant, the coefficient $c_H$ is zero.  We compute this
coefficient by summing over all copies of a three-edge star in $\UG(T)$.
The only possible centers of such copies are $u,x,y$.  By Lemma
\ref{lem:constantsignstar}, the sum over the isomorphisms of a fixed copy is a
constant independent of the order in which the leaves are mapped.  Thus, after
multiplication by one fixed non-zero constant, the coefficient is the sum of
the signed numbers of three-edge stars centered at $u,x,y$.

We use the following notation.  If a vertex has total degree $d$ and exactly
$i$ incident edges pointing into it, then the signed contribution of the
three-edge stars centered at that vertex is
$$
        K_3^{(d)}(i)
        =
        \sum_{q=0}^3(-1)^q\binom{i}{q}\binom{d-i}{3-q}.
$$
In the present case the total degrees and incoming degrees are
$$
        d(u)=2h+2,
        \qquad
        i(u)=h,
$$
and
$$
        d(x)=d(y)=2h+1,
        \qquad
        i(x)=2h+1-o_x,
        \qquad
        i(y)=2h+1-o_y.
$$
Since $o_y=2h-o_x$, putting $a=o_x$ gives
$$
        i(x)=2h+1-a,
        \qquad
        i(y)=a+1.
$$
Therefore
$$
        0=c_H/C
        =K_3^{(2h+2)}(h)
        +K_3^{(2h+1)}(2h+1-a)
        +K_3^{(2h+1)}(a+1),
$$
where $C\neq0$.  A direct calculation gives
$$
        K_3^{(2h+2)}(h)
        +K_3^{(2h+1)}(2h+1-a)
        +K_3^{(2h+1)}(a+1)
        =-4(a-h)^2.
$$
Hence $a=h$.  Thus $o_x=h$, and since $o_x+o_y=2h$, also $o_y=h$.
This proves the lemma.
\end{proof}

\begin{proposition}\label{prop:diam4-remaining-case}
Let $T$ be a converse invariant orientation of a non-path tree of diameter
four, and suppose that
$$
        \sum_{r\geq1}|S_r|=2.
$$
Then either $T$ is self-converse, or the following exceptional situation
holds.  There is an even integer $t=2h$ such that
$$
        S_t=\{x,y\},
        \qquad
        |S_0|=t,
$$
every one of the three vertices $u,x,y$ has exactly $h$ outgoing pendant
edges and exactly $h$ incoming pendant edges, and the two central edges are
both directed away from $u$, or both directed towards $u$.
\end{proposition}

\begin{proof}
Let $S_t=\{x,y\}$, and put $s=|S_0|$.  By Lemma
\ref{lem:diam4-two-branch-equations}, we have
$$
        2o_u=s
        \qquad\text{and}\qquad
        o_x+o_y=t.
$$

Assume first that the two central edges have opposite directions, say
$u\to x$ and $y\to u$.  Then the branch at $x$ belongs to
$A_{t-o_x}^t$, while the branch at $y$ belongs to $B_{o_y}^t$.  The
identity $o_x+o_y=t$ is precisely
$$
        t-o_x=o_y.
$$
Thus the two branches in $S_t$ are paired by converse.  Also
$2o_u=s$, so the leaves in $S_0$ are paired by converse.  Hence $T$ is
self-converse.

Consequently, if $T$ is not self-converse, the two central edges have the
same direction, unless this is ruled out by the coefficient computation.  If
$s\neq t$, Lemma \ref{lem:diam4-two-branch-equations} gives
$\epsilon_x+\epsilon_y=0$, so the central edges have opposite directions,
and we are back in the self-converse case.  Hence every non-self-converse
example must satisfy $s=t$ and must have the two central edges in the same
direction.  Lemma \ref{lem:diam4-symmetric-two-branch} then gives
$$
        t=2h,
        \qquad
        o_u=o_x=o_y=h.
$$
This is exactly the exceptional configuration stated above.
\end{proof}

\begin{lemma}\label{lem:diam4-exception-converse-invariant}
The exceptional configuration of Proposition \ref{prop:diam4-remaining-case}
is converse invariant.
\end{lemma}

\begin{proof}
Let $E_h^+$ be the exceptional tree in which the middle center is a source,
and let $E_h^-=\overline{E_h^+}$ be its converse, in which the middle center
is a sink.  Thus $E_h^+$ is obtained from a path
$$
        x-u-y
$$
by attaching $2h$ leaves to each of $x,u,y$, orienting exactly $h$ pendant
edges away from each center and exactly $h$ pendant edges towards each
center, and orienting the two central edges as
$$
        x\leftarrow u\to y.
$$
The converse $E_h^-$ has the same balanced pendant stars, but its central
edges are oriented as
$$
        x\to u\leftarrow y.
$$

It is enough to compare labeled embeddings of $E_h^+$ and $E_h^-$ into an
arbitrary tournament $T$, since $E_h^+$ and $E_h^-$ have the same number of
automorphisms.  This is the same passage from embeddings to copies used in
the definition of $f_T(D)$ in Section~\ref{sec:prelim}.

Fix a set $C=\{a,b,c\}$ of three vertices of $T$.  We shall compare the
embeddings whose three centers are mapped precisely to the vertices of $C$.
Call a choice of pendant leaves around $C$ \emph{balanced} if, for each
$v\in C$, we choose $h$ vertices which receive an edge from $v$ and $h$
vertices which send an edge to $v$, all these chosen vertices being distinct
and lying outside $C$.  Such a balanced choice determines the pendant part of
an embedding, up to the same leaf-labeling factor for $E_h^+$ and $E_h^-$.
In particular, for the comparison below, the number of balanced pendant
choices depends only on $T$ and on $C$, and not on whether we are embedding
$E_h^+$ or $E_h^-$.

Now consider the tournament induced by $C$.  There are two cases.

First suppose that $T[C]$ is cyclic.  Then no vertex of $C$ dominates the
other two vertices, and no vertex of $C$ is dominated by the other two
vertices.  Hence no vertex of $C$ can play the role of the middle source in
$E_h^+$, and no vertex of $C$ can play the role of the middle sink in
$E_h^-$.  Therefore the set $C$ contributes zero embeddings of $E_h^+$ and
zero embeddings of $E_h^-$.

Now suppose that $T[C]$ is transitive.  Let $s(C)$ be its source and let
$t(C)$ be its sink.  Then $s(C)$ is the unique vertex of $C$ which dominates
the other two vertices, and $t(C)$ is the unique vertex of $C$ which is
dominated by the other two vertices.

For an embedding of $E_h^+$ with center set $C$, the middle center $u$ must
be mapped to $s(C)$, since $u$ must point to both endpoint centers.  The two
endpoint centers $x$ and $y$ may then be mapped to the two remaining vertices
of $C$ in either order.  Thus, after fixing the balanced pendant choices,
there are exactly two possible center maps for $E_h^+$.

Similarly, for an embedding of $E_h^-$ with center set $C$, the middle
center $u$ must be mapped to $t(C)$, since both endpoint centers must point
towards $u$.  Again the two endpoint centers may be mapped to the two
remaining vertices of $C$ in either order.  Thus, after fixing the same
balanced pendant choices, there are exactly two possible center maps for
$E_h^-$.

Consequently, every transitive triple $C$ contributes the same number of
labeled embeddings of $E_h^+$ and $E_h^-$, while every cyclic triple
contributes zero to both.  Summing over all three-element sets
$C\subseteq V(T)$ gives
$$
        \emb_T(E_h^+)=\emb_T(E_h^-).
$$
Since $|\Aut(E_h^+)|=|\Aut(E_h^-)|$, it follows that
$$
        f_T(E_h^+)=f_T(E_h^-)
$$
for every tournament $T$.  Hence the exceptional configuration is converse
invariant.
\end{proof}

\begin{proposition}
    The trees in $\mathcal E_4$ are converse invariant but not self-converse and cannot be obtained by the bridge-mirroring operation recursively from an orientation of a path.
\end{proposition}

\begin{proof}
Let $T\in\mathcal E_4$. Then $T$ has
$$
        3(2h+1)=6h+3
$$
vertices, which is odd. Also $T$ is not a path and has maximum degree $2h+2\geq3$.

If a tree is obtained from a path by at least one bridge-mirroring operation, then its
number of vertices is even, since each bridge-mirroring step takes two disjoint copies of
the previous tree and adds one bridge edge between them. Since $T$ is not a path, any
recursive bridge-mirroring construction of $T$ would have to use at least one such step,
and hence would produce a tree with an even number of vertices. This is impossible.

By Lemma \ref{lem:diam4-exception-converse-invariant}, the tree $T$ is converse
invariant. Moreover, $T$ is not self-converse: the middle vertex is the unique vertex of
degree $2h+2$, so every isomorphism must fix it, but reversing all arcs changes it from
a source on the central path to a sink, or conversely. Therefore $T$ is a
non-self-converse converse-invariant tree which is not obtained by recursive
bridge-mirroring from a path.
\end{proof}

Putting everything together, we prove the main result of the paper:

\begin{proof}[Proof of Theorem \ref{thm:diam4classification}]
Let $T$ be a converse invariant orientation of a tree of diameter four
which is not a path.  If
$$
        \sum_{t\geq 1}|S_t|\geq 4,
$$
then Theorem \ref{thm:largecore} implies that $T$ is self-converse.

It remains to consider the case
$$
        \sum_{t\geq 1}|S_t|<4.
$$
By Lemma \ref{lem:remaining-shape}, there is a unique integer $t\geq 1$
such that $S_t=\{x,y\}$.  Proposition
\ref{prop:diam4-remaining-case} then shows that either $T$ is
self-converse, or $T\in\mathcal E_4$.

Conversely, every self-converse orientation is converse invariant, and
every tree in $\mathcal E_4$ is converse invariant by Lemma
\ref{lem:diam4-exception-converse-invariant}.  This proves the theorem.
\end{proof}

\section{A new family of non-self-converse converse invariant digraphs}

It is possible to generalize the construction of Definition \ref{def:exceptionaldiam4} of non-self-converse converse invariant trees using an operation we call \emph{rooted substitution}. This operation provides a new way of creating converse invariant digraphs which are not self-converse. In particular, it provides counterexamples to Conjecture \ref{conj:ai} of arbitrarily large diameter:

\begin{definition}\label{def:rooted-substitution}
Let $D$ be an oriented graph, and let $(R,r)$ be a rooted oriented graph.
We write $D[R]$, the \emph{rooted substitution} of $R$ in $D$, for the oriented graph obtained as follows.  For each
vertex $v\in V(D)$, take a copy $R_v$ of $R$, and identify the root
$r$ of $R_v$ with $v$.  The edges between the roots are the original
edges of $D$, and the edges inside each copy $R_v$ are oriented as in
$R$.
\end{definition}

\begin{definition}\label{def:rooted-self-converse}
A rooted oriented graph $(R,r)$ is called \emph{rooted self-converse} if there
is an isomorphism
$$
        \varphi:R\to \overline R
$$
such that $\varphi(r)=r$.
\end{definition}

\begin{proposition}\label{prop:rooted-substitution}
Let $D$ be a converse invariant oriented graph, and let $(R,r)$ be a
rooted self-converse oriented graph.  Then $D[R]$ is converse invariant.
In particular, if $D$ and $R$ are trees, then $D[R]$ is a converse
invariant tree. 
\end{proposition}

\begin{proof}
Let $U$ be an arbitrary tournament.  We count labeled embeddings; this is
equivalent to counting copies, since $D[R]$ and $\overline{D[R]}$ have
the same automorphism group.

Let $C\subseteq V(U)$ be a set of size $|V(D)|$.  We think of $C$ as the
set of images of the vertices of $D$, or equivalently, as the set of
roots of the attached copies of $R$.  For $c\in C$, let
$W_c\subseteq V(U)\setminus C$ be the set of vertices used for the
non-root vertices of the copy of $R$ attached at $c$.  The sets $W_c$ are
required to be pairwise disjoint and to have size $|V(R)|-1$.

For a fixed family $(W_c)_{c\in C}$, let
$$
        N_U(c,W_c)
$$
be the number of rooted embeddings of $(R,r)$ into $U[\{c\}\cup W_c]$
which send $r$ to $c$.  Since $(R,r)$ is rooted self-converse, we also
have
$$
        N_U(c,W_c)
        =
        N_U^{\,\overline R}(c,W_c),
$$
where the right-hand side denotes the number of rooted embeddings of
$(\overline R,r)$ into the same tournament, again sending the root to
$c$.

Define
$$
        W_U(C)
        =
        \sum_{(W_c)_{c\in C}}
        \prod_{c\in C} N_U(c,W_c),
$$
where the sum runs over all pairwise disjoint choices of the sets
$W_c\subseteq V(U)\setminus C$ with $|W_c|=|V(R)|-1$.  The number
$W_U(C)$ depends only on the set $C$, not on any particular bijection
from $V(D)$ to $C$.

Now fix $C$.  Once the root set is $C$, the choices inside the attached
copies contribute the common factor $W_U(C)$.  The only remaining choice
is the embedding of the base graph $D$ into the tournament $U[C]$.
Therefore the number of labeled embeddings of $D[R]$ into $U$ with root
set $C$ is
$$
        W_U(C)\,\emb_{U[C]}(D).
$$
Similarly, since $\overline{D[R]}=\overline D[\overline R]$ and
$(R,r)$ is rooted self-converse, the number of labeled embeddings of
$\overline{D[R]}$ into $U$ with root set $C$ is
$$
        W_U(C)\,\emb_{U[C]}(\overline D).
$$

Since $D$ is converse invariant, we have
$$
        \emb_{U[C]}(D)=\emb_{U[C]}(\overline D)
$$
for every tournament $U[C]$.  Hence the two expressions above are equal
for every choice of $C$.  Summing over all $C\subseteq V(U)$ with $|C|=|V(D)|$ gives
$$
        \emb_U(D[R])=\emb_U(\overline{D[R]}).
$$
Thus $D[R]$ is converse invariant.
\end{proof}

\bibliographystyle{plain}
{\footnotesize

  \bibliography{counting-tournaments.bib}

@article{ai2025number,
  title={Number of Subgraphs and Their Converses in Tournaments and New Digraph Polynomials},
  author={Ai, Jiangdong and Gutin, Gregory and Lei, Hui and Yeo, Anders and Zhou, Yacong},
  journal={Journal of Graph Theory},
  year={2025},
  publisher={Wiley Online Library}
}

@article{el2023number,
  title={About the number of oriented Hamiltonian paths and cycles in tournaments},
  author={El Sahili, Amine and Ghazo Hanna, Zeina},
  journal={Journal of Graph Theory},
  volume={102},
  number={4},
  pages={684--701},
  year={2023},
  publisher={Wiley Online Library}
}

@article{zhao2020impartial,
  title={Impartial digraphs},
  author={Zhao, Yufei and Zhou, Yunkun},
  journal={Combinatorica},
  volume={40},
  number={6},
  pages={875--896},
  year={2020},
  publisher={Springer}
}




}
\end{document}